\documentclass[12pt]{article}

\usepackage{amsmath, amssymb, amsthm}

\theoremstyle{plain}
\newtheorem{theorem}{Theorem}
\newtheorem{lemma}[theorem]{Lemma}
\newtheorem{corollary}[theorem]{Corollary}
\usepackage{fullpage}

\DeclareMathOperator{\ex}{ex}
\DeclareMathOperator{\exs}{exs}
\def\modd#1 #2{#1\ \mbox{{\rm (mod}} \ #2\mbox{\rm )}}
\def\nodiv{{|\kern-3.5pt/}}
\def\Zee{{\mathbb{Z}}}

\begin{document}

\title{A Class of Exponential Sequences with Shift-Invariant Discriminators}

\author{Sajed Haque and Jeffrey Shallit\\
School of Computer Science \\
University of Waterloo \\
Waterloo, ON  N2L 3G1 \\
Canada \\
{\tt s24haque@cs.uwaterloo.ca} \\
{\tt shallit@cs.uwaterloo.ca} 
}

\date{\today}

\maketitle

\begin{abstract}
The discriminator of an integer sequence $\textbf{s} = (s(i))_{i \geq 0}$, introduced by Arnold, Benkoski, and McCabe in 1985, is the function $D_{\textbf{s}} (n)$ that sends $n$ to the least integer $m$ such that the numbers $s(0), s(1), \ldots, s(n - 1)$ are pairwise incongruent modulo $m$.  In this note we present a class of exponential sequences that have the special property that their discriminators are shift-invariant, i.e., that the discriminator of the sequence is the same even if the sequence is shifted by any positive constant.
\end{abstract}

\section{Discriminators}
\label{intro}

Let $m$ be a positive integer.  If $S$ is a set of integers that are pairwise incongruent modulo $m$, we say that $m$ \textit{discriminates} $S$. Now let $\textbf{s} = (s(i))_{i \geq 0}$ be a
sequence of distinct integers. For all integers $n \geq 1$, we define $D_\textbf{s} (n)$ to be the least positive integer $m$ that discriminates the set $\{s(0), s(1), \ldots, s(n - 1)\}$.  The function $D_\textbf{s} (n)$ is called the \textit{discriminator} of the sequence \textbf{s}. 

The discriminator was first introduced by Arnold, Benkoski, and McCabe \cite{ABM85}. They derived the discriminator for the sequence $1, 4, 9, \ldots$ of positive integer squares. More recently, discriminators of various sequences were studied by Schumer and Steinig \cite{SchumerSteinig88}, Barcau \cite{Barcau88}, Schumer \cite{Schumer90}, Bremser, Schumer, and Washington \cite{BSW90},
Moree and Roskam \cite{MoreeRoskam95}, Moree \cite{Moree96}, Moree and Mullen \cite{MoreeMullen96}, Zieve \cite{Zieve98}, Sun \cite{Sun13},
Moree and Zumalac\'arrequi \cite{MoreeZumalacarregui16}, and Haque and Shallit \cite{HaqueShallit16}.

In all of these cases, however, the discriminator is based on the first $n$ terms of a sequence, for $n \geq 2$. Therefore, the discriminator can depend
crucially
on the starting point of a given sequence. For example, although the discriminator for the first three positive squares, $\{1, 4, 9 \}$, is $6$, we can see that the number $6$ does not discriminate the length-3 ``window'' into the
shifted sequence, $\{4, 9, 16 \}$, since $16 \equiv \modd{4} {6}$. 

Furthermore, there has been very little work on the discriminators of exponential sequences. Sun \cite{Sun13} presented some conjectures concerning
certain exponential sequences, while in a recent {\it tour de force},
Moree and  Zumalac\'arrequi \cite{MoreeZumalacarregui16} computed the discriminator for the sequence $\left( \frac{|(-3)^j - 5|}{4} \right)_{j \geq 0}$.

We say that the discriminator of a sequence is \textit{shift-invariant} if the discriminator for the sequence is the same even if the sequence is shifted by any positive integer $c$, i.e., for all positive integers $c$
the discriminator of the sequence
$(s(n))_{n \geq 1}$ is the same as the discriminator of the
sequence $(s(n + c))_{n \geq 0}$.
In this paper, we present a class of exponential sequences whose discriminators are shift-invariant.

We define this class of exponential sequences as follows:
$$(\ex(n))_{n \geq 0} = \left(a \frac{(t^2)^n - 1}{2^b}\right)_{n \geq 0}$$
for odd positive integers $a$ and $t$, where $b$ is the smallest positive integer such that $t \not\equiv \modd{\pm 1} {2^b}$. A typical example is the sequence $\left(\frac{9^n - 1}{8}\right)_{n \geq 0}$. We show that the discriminator for all sequences of this form is $D_{\ex} (n) = 2^{\lceil \log_2 n \rceil}$. Furthermore, we show that this discriminator is shift-invariant, i.e., it applies to every sequence $(\ex(n + c))_{n \geq 0}$ for $c \geq 0$. 

The outline of the paper is as follows.  In Section~\ref{two} we obtain an
upper bound for the discriminator
of $\left(\frac{(t^2)^n - 1}{2^b}\right)_{n \geq 0}$ and all of its shifts.
In Section~\ref{three} we prove some lemmas that are essential to our
lower bound proof.  Finally, in Section~\ref{four} we put the results together
to determine the discriminator for $ \left(a \frac{(t^2)^n - 1}{2^b}\right)_{n \geq 0}$ and all of its shifts.

\section{Upper bound}
\label{two}

In this section, we derive an upper bound for the 
discriminator of the sequence
$\left(\frac{(t^2)^n - 1}{2^b}\right)_{n \geq 0}$ and all of its
shifts.  We start with
some useful lemmas.

\begin{lemma}
\label{tsqrmod2pow}
Let $t$ be an odd integer, and let $b$ be the smallest positive integer such that $t \not\equiv \modd{\pm 1} {2^b}$. Then $t^2 \equiv \modd{2^b + 1} {2^{b + 1}}$.
\end{lemma}
\begin{proof}
Note that since every odd integer equals $\pm 1$ modulo 4,
we must have $b \geq 3$. From the definition of $b$, we have $t \equiv \modd{2^{b - 1} \pm 1} {2^b}$.  Hence $t = 2^b c + 2^{b - 1} \pm 1$ for some integer $c$. By squaring both sides of the equation, we get
\begin{align*}
t^2 &= 2^{2b} c^2 + 2^{2(b - 1)} + 2^{2b} c \pm 2^{b + 1} c \pm 2^b + 1\\
&= 2^{b + 1} \left(2^{b - 1} c^2 + 2^{b - 3} + 2^{b - 1} c \pm c \right) \pm 2^b + 1,\\
\implies t^2 &\equiv \modd{\pm 2^b + 1} {2^{b + 1}},\\
\implies t^2 &\equiv \modd{2^b + 1} {2^{b + 1}}.
\end{align*}
\end{proof}

\begin{lemma}
Let $t$ be an odd integer, and let $b$ be the smallest positive integer such that $t \not\equiv \modd{\pm 1} {2^b}$. Then we have
\begin{equation}
\label{tpowmod2pow}
t^{2^k} \equiv \modd{2^{k + b - 1} + 1} {2^{k + b}}
\end{equation}
for all integers $k \geq 1$.
\end{lemma}

\begin{proof}
By induction on $k$.
\begin{description}
\item[Base case:] From Lemma~\ref{tsqrmod2pow}, we have $t^2 \equiv \modd{2^b + 1} {2^{b + 1}}$. 
\item[Induction:] Suppose Eq.~\eqref{tpowmod2pow} holds for some $k \geq 1$, i.e., $t^{2^k} \equiv \modd{2^{k + b - 1} + 1} {2^{k + b}}$. This means that $t^{2^k} = 2^{k + b} c + 2^{k + b - 1} + 1$ for some integer $c$. Once again, by squaring both sides of the equation, we get
\begin{align*}
\left(t^{2^k}\right)^2 = t^{2^{k + 1}} &= 2^{2k + 2b} c^2 + 2^{2k + 2b - 2} + 1 + 2^{2k + 2b} c + 2^{k + b + 1} c + 2^{k + b}\\
&= 2^{k + b + 1} \left(2^{k + b - 1} c^2 + 2^{k + b - 3} + 2^{k + b - 1} c + c \right) + 2^{k + b} + 1,\\
\implies t^{2^{k + 1}} &\equiv \modd{2^{k + b} + 1} {2^{k + b + 1}}.
\end{align*}
This shows that Eq.~\eqref{tpowmod2pow} holds for $k + 1$ as well, thus completing the induction.
\end{description}
\end{proof}

This gives the following corollary.
\begin{corollary}
\label{tsqrordermod2pow}
Let $t$ be an odd integer, and let $b$ be the smallest positive integer such that $t \not\equiv \modd{\pm 1} {2^b}$. Then for $k \geq 1$, the powers of $t^2$ form a cyclic subgroup of order $2^k$ in $(\Zee/2^{k + b})^*$.
\end{corollary}
\begin{proof}
Let $\ell = k + 1$. Since $\ell \geq 1$, we can apply Eq.~\eqref{tpowmod2pow} to get 
\begin{align*}
(t^2)^{2^{\ell - 1}} = t^{2^\ell} &\equiv \modd{2^{\ell + b - 1} + 1} {2^{\ell + b}},\\
\implies (t^2)^{2^{\ell - 1}} &\equiv \modd{1} {2^{\ell + b - 1}},\\
\implies (t^2)^{2^k} &\equiv \modd{1} {2^{k + b}}.
\end{align*}
Furthermore, by applying Eq.~\eqref{tpowmod2pow} directly, we get
\begin{align*}
(t^2)^{2^{k - 1}} = t^{2^k} &\equiv 2^{k + b - 1} + 1 \not\equiv \modd{1} {2^{k + b}},\\
\implies (t^2)^{2^{k - 1}} &\not\equiv \modd{1} {2^{k + b}} .
\end{align*} 
Therefore, the order of the subgroup generated by $t^2$ in $(\Zee/2^{k + b})^*$ is $2^k$.
\end{proof}

\begin{lemma}
\label{exupperbound}
Let $t$ be an odd integer, and let $b$ be the smallest positive integer such that $t \not\equiv \modd{\pm 1} {2^b}$. Then for $k \geq 0$, the number $2^k$ discriminates every set of $2^k$ consecutive terms of the sequence $\left(\frac{(t^2)^n - 1}{2^b}\right)_{n \geq 0}$.
\end{lemma}

\begin{proof}
For every $i \geq 0$, it follows from Corollary~\ref{tsqrordermod2pow} that the numbers
$$(t^2)^i, (t^2)^{i + 1},\ldots, (t^2)^{i + 2^k - 1}$$
are distinct modulo $2^{k + b}$. By subtracting 1 from every element, we have that the numbers 
$$(t^2)^i - 1, (t^2)^{i + 1} - 1, \ldots, (t^2)^{i + 2^k - 1} - 1$$ are distinct modulo $2^{k + b}$. 
Furthermore, these numbers are also congruent to 0 modulo $2^b$ because $t^2 \equiv \modd{1} {2^b}$ from Lemma~\ref{tsqrmod2pow}. It follows that the set of
quotients 
$$\left\{\frac{(t^2)^i - 1}{2^b}, \frac{(t^2)^{i + 1} - 1}{2^b}, \ldots, \frac{(t^2)^{i + 2^k - 1} - 1}{2^b}\right\}$$ consists of integers that are distinct modulo $\frac{2^{k + b}}{2^b} = 2^k$. 

Such a set of quotients coincides with every set of 
$2^k$ consecutive 
terms of the sequence $\left(\frac{(t^2)^n - 1}{2^b}\right)_{n \geq 0}$.
Since the numbers in each set are distinct modulo $2^k$, the desired result follows. 
\end{proof}

\section{Lower bound}
\label{three}

In this section, we establish some results useful for the lower bound
on the discriminator of
the sequence
$\left(\frac{(t^2)^n - 1}{2^b}\right)_{n \geq 0}$.  We start with an
easy technical lemma, whose proof is omitted.

\begin{lemma}
\label{log3mleqmover3}
Let $m$ be a positive integer. Then $\log_3 m \leq \frac{m}{3}$.
\end{lemma}

The main lemma for proving the lower bound is as follows:

\begin{lemma}
\label{exlowerbound}
Let $t$ be an odd integer, and let $b$ be the smallest positive integer such that $t \not\equiv \modd{\pm 1} {2^b}$. Then for all  $k \geq 0$ and $1 \leq m \leq 2^{k + 1}$, there exists a pair of integers, $i$ and $j$, where $0 \leq i < j \leq 2^k$, such that $(t^2)^i \equiv \modd{(t^2)^j} {2^b m}$. 
\end{lemma}

\begin{proof}
Let the prime factorization of $m$ be
\begin{align*}
m &= 2^x \prod_{1 \leq \ell \leq u} p_\ell^{y_\ell} \prod_{1 \leq \ell \leq v} q_\ell^{z_\ell},
\end{align*}
where $u, v, x, y_\ell, z_\ell \geq 0$, while $p_1, p_2, \ldots, p_u$ are the prime factors of $m$ that also divide $t$, and $q_1, q_2, \ldots, q_v$ are the odd prime factors of $m$ that do not divide $t$. For each $\ell \leq u$, let $e_\ell$ be the integer such that $p_\ell^{e_\ell} || t$, i.e., we have $p_\ell^{e_\ell} | t$ but $p_\ell^{e_\ell + 1} \nmid t$.

We need to find a pair $(i, j)$ such that $(t^2)^i \equiv \modd{(t^2)^j} {2^b m}$. From the Chinese remainder theorem, we know it suffices to find a pair $(i, j)$ such that
\begin{align*}
(t^2)^i &\equiv \modd{(t^2)^j} {2^{x + b}},\\
(t^2)^i &\equiv \modd{(t^2)^j} {p_\ell^{y_\ell}}, \text{ for all }1 \leq \ell \leq u,\\
\text{and } (t^2)^i &\equiv \modd{(t^2)^j} {q_\ell^{z_\ell}}, \text{ for all }1 \leq \ell \leq v.
\end{align*}

For the first of these equations, we know from Corollary~\ref{tsqrordermod2pow} that $(t^2)^i \equiv \modd{(t^2)^{i + 2^x}} {2^{x + b}}$. In other words, it suffices to have $2^x | (j - i)$ to satisfy $(t^2)^i \equiv \modd{(t^2)^j} {2^{x + b}}$.

Next, we consider the $u$ equations of the form $(t^2)^i \equiv \modd{(t^2)^j} {p_\ell^{y_\ell}}$. Since $p_\ell^{e_\ell}$ is a factor of $t$, it follows that $(t^2)^{y_\ell / 2 e_\ell}$ is a multiple of $(p_\ell^{2e_\ell})^{y_\ell / 2 e_\ell} = p_\ell^{y_\ell}$. Therefore, $(t^2)^{y_\ell / 2 e_\ell} \equiv \modd {0} {p_\ell^{y_\ell}}$. Any further multiplication by $t^2$ also yields 0 modulo $p_\ell^{y_\ell}$. Thus, it suffices to have $j > i \geq \frac{y_\ell}{2 e_\ell}$ in order to ensure that $(t^2)^i \equiv \modd{(t^2)^j} {p_\ell^{y_\ell}}$.

Finally, there are $v$ equations of the form $(t^2)^i \equiv \modd{(t^2)^j} {q_\ell^{z_\ell}}$. In each case, $q_\ell$ is co-prime to $t$, which means that $(t^2)^{\varphi(q_\ell^{z_\ell})/2} = t^{\varphi(q_\ell^{z_\ell})} \equiv \modd {1} {q_\ell^{z_\ell}}$, where $\varphi(n)$ is Euler's totient function. Now $\frac{\varphi (q_\ell^{z_\ell})}{2} = \frac{q_\ell^{z_\ell - 1}(q_\ell - 1)}{2}$. Thus, it is sufficient to have $\frac{q_\ell^{z_\ell - 1}(q_\ell - 1)}{2} | (j - i)$ in order to ensure that $(t^2)^i \equiv \modd{(t^2)^j} {q_\ell^{z_\ell}}$.

Merging these ideas together, we choose the following values for $i$ and $j$:
\begin{align*}
i &= \max_{1 \leq \ell \leq u} \left\lceil \frac{y_\ell}{2 e_\ell} \right\rceil,\\
j &= \max_{1 \leq \ell \leq u} \left\lceil \frac{y_\ell}{2 e_\ell} \right\rceil + 2^x \prod_{1 \leq \ell \leq v} \frac{q_\ell^{z_\ell - 1}(q_\ell - 1)}{2},
\end{align*}
to ensure that $(t^2)^i \equiv \modd{(t^2)^j} {2^b m}$. It is clear that $0 \leq i < j$. In order to show that $j \leq 2^k$, we first observe that
\begin{align*}
j &= \max_{1 \leq \ell \leq u} \left\lceil \frac{y_\ell}{2 e_\ell} \right\rceil + 2^x \prod_{1 \leq \ell \leq v} \frac{q_\ell^{z_\ell - 1}(q_\ell - 1)}{2} = \max_{1 \leq \ell \leq u} \left\lceil \frac{y_\ell}{2 e_\ell} \right\rceil + \frac{2^x}{2^v} \prod_{1 \leq \ell \leq v} q_\ell^{z_\ell - 1}(q_\ell - 1)\\
& \leq \max_{1 \leq \ell \leq u} \left\lceil \frac{y_\ell}{2} \right\rceil + \frac{2^x}{2^v} \prod_{1 \leq \ell \leq v} q_\ell^{z_\ell} = \max_{1 \leq \ell \leq u} \left\lceil \frac{y_\ell}{2} \right\rceil + \frac{m}{2^v \prod_{1 \leq \ell \leq u} p_\ell^{y_\ell}}.
\end{align*}
We now consider the following two cases:
\begin{description}
\item[Case 1: $u = 0$.] If $v = 0$ as well, then $j = 2^x = m < 2^{k + 1}$, which means that $x \leq k$ and thus $j \leq 2^k$. Otherwise, if $v \geq 1$, then we have
\begin{equation*}
j \leq \max_{1 \leq \ell \leq u} \left\lceil \frac{y_\ell}{2} \right\rceil + \frac{m}{2^v \prod_{1 \leq \ell \leq u} p_\ell^{y_\ell}} = \frac{m}{2^v} \leq \frac{m}{2} < \frac{2^{k + 1}}{2} = 2^k.
\end{equation*}
\item[Case 2: $u \geq 1$.] Let $r$ be such that $y_r = \max_{1 \leq \ell \leq u} y_\ell$, and thus, $p_r$ is the corresponding prime number with exponent $y_r$. Since $p_r^{y_r} \geq p_r \geq 3$, we have
\begin{equation*}
j \leq \max_{1 \leq \ell \leq u} \left\lceil \frac{y_\ell}{2} \right\rceil + \frac{m}{2^v \prod_{1 \leq \ell \leq u} p_\ell^{y_\ell}} \leq \left\lceil \frac{y_r}{2} \right\rceil + \frac{m}{p_r^{y_r}} \leq \frac{y_r + 1}{2} + \frac{m}{3} \leq \frac{y_r}{2} + \frac{1}{2} + \frac{m}{3}.
\end{equation*}
Note that $y_r \leq \log_{p_r} m \leq \log_{3} m \leq \frac{m}{3}$ from Lemma~\ref{log3mleqmover3}, which means that
\begin{equation*}
j \leq \frac{y_r}{2} + \frac{1}{2} + \frac{m}{3} \leq \frac{m}{6} + \frac{1}{2} + \frac{m}{3} = \frac{m}{2} + \frac{1}{2} = \frac{m + 1}{2}.
\end{equation*}
Since both $m$ and $j$ are integers, this implies that
\begin{equation*}
j \leq \left\lceil \frac{m}{2} \right\rceil \leq \left\lceil \frac{2^{k + 1}}{2} \right\rceil \leq 2^k.
\end{equation*}
\end{description}
In both cases, we have $j \leq 2^k$, thus fulfilling the required conditions.
\end{proof}

\section{Discriminator of $(\ex(n))_{n \geq 0}$ and its shifted counterparts}
\label{four}

In this section, we combine the results of the previous sections to determine the discriminator for $(\ex(n))_{n \geq 1}$, as well as its shifted counterparts. We first prove a general lemma about the discriminator of some scaled sequences.

\begin{lemma}
\label{coprimescaledisc}
Given a sequence $s(0), s(1), \ldots,$ and a non-zero integer $a$, let $s'(0), s'(1), \ldots,$ denote the sequence such that $s'(i) = a s(i)$ for all $i \geq 0$. Then, for every $n$ such that $\gcd(|a|, D_s (n)) = 1$, we have $D_{s'} (n) = D_s (n)$.
\end{lemma}

\begin{proof}
From the definition of the discriminator, we know that for every $m < D_s (n)$, there exists a pair of integers $i$ and $j$ with $i < j < n$, such that $m | s(j) - s(i)$. Thus, for this same pair of $i$ and $j$, we have 
\begin{equation*}
m | a(s(j) - s(i)) = as(j) - as(i) = s'(j) - s'(i).
\end{equation*}
Therefore, $m$ cannot discriminate the set $\{s'(0), s'(1), \ldots, s'(n - 1)\}$ and so $D_{s'} (n) \geq D_s (n)$.

But for $m = D_s (n)$, we know that for all $i$ and $j$ with $i < j < n$, we have $m \nmid s(j) - s(i)$. Since $\gcd (m, |a|) = 1$, it follows that 
\begin{equation*}
m \nmid a (s(j) - s(i)) = as(j) - as(i) = s'(j) - s'(i)
\end{equation*}
for all $i$ and $j$ with $i < j < n$. Therefore,
$m = D_s (n)$ discriminates the set 
$$\{s'(0), s'(1), \ldots, s'(n - 1)\}$$ 
and so $D_{s'} (n) \leq D_s (n)$. 

Putting these results together, we have $D_{s'} (n) = D_s (n)$.
\end{proof}

We now compute the discriminator for $(\ex(n))_{n \geq 0} = \left(a \frac{(t^2)^n - 1}{2^b}\right)_{n \geq 0}$, and also for its shifted counterparts, which we denote by $(\exs(n, c))_{n \geq 0} = (\ex(n + c))_{n \geq 0}$ for some integer $c \geq 0$.

\begin{theorem}
Let $t$, $a$, $b$, and $c$ be integers such that $a$ and $t$ are odd, $c \geq 0$, and let $b$ be the smallest integer such that $t \not\equiv \modd{\pm 1} {2^b}$. Then the discriminator for the sequence $(\exs(n, c))_{n \geq 0} = \left(a \frac{(t^2)^{n + c} - 1}{2^b}\right)_{n \geq 0}$ is
\begin{equation}
D_{\exs} (n) = 2^{\lceil \log_2 n \rceil}.
\end{equation}

\begin{proof}
First we compute the discriminator for $a = 1$, where the sequence is of the form $(\exs(n))_{n \geq 0} = \left(\frac{(t^2)^{n + c} - 1}{2^b}\right)_{n \geq 0}$. 

The case for $n = 1$ is trivial. Otherwise, let $k \geq 0$ be such that $2^k < n \leq 2^{k + 1}$. We show that $D_{\exs} (n) = 2^{k + 1}$. 

From Lemma~\ref{exupperbound}, we know that $2^{k + 1}$ discriminates the set,
$$\{\ex(c), \ex(c + 1), \ldots, \ex (c + 2^{k + 1} - 1)\},$$
as well as every smaller subset of these numbers. Therefore, $2^{k + 1}$ 
discriminates 
$$\{\exs (0, c), \exs (1, c), \ldots, \exs (n - 1, c)\}.$$ 
In other words, $D_{\exs}(n) \leq 2^{k + 1}$.

Now let $m$ be a positive integer such that $m < 2^{k + 1}$. By Lemma~\ref{exlowerbound}, we know that there exists a pair of integers, $i$ and $j$, such that
\begin{align*}
(t^2)^i \equiv \modd{(t^2)^j} {2^b m} &\implies (t^2)^c (t^2)^i \equiv \modd{(t^2)^c (t^2)^j} {2^b m},\\
&\implies (t^2)^{i + c} - 1 \equiv \modd{(t^2)^{j + c} - 1} {2^b m}.
\end{align*}

Note that since $(t^2) \equiv \modd{1} {2^b}$ from Lemma~\ref{tsqrmod2pow}, we have $(t^2)^{i + c} - 1 \equiv (t^2)^{j + c} - 1 \equiv 1 - 1 \equiv \modd{0} {2^b}$. Therefore,
\begin{equation*}
(t^2)^{i + c} - 1 \equiv \modd{(t^2)^{j + c} - 1} {2^b m} \implies \frac{(t^2)^{i + c} - 1}{2^b} \equiv \modd{\frac{(t^2)^{j + c} - 1}{2^b}} {m}.
\end{equation*}

In other words, $\exs(i, c) \equiv \modd {\exs(j, c)} {m}$ while both numbers are in the set $$\{\exs (0, c), \exs (1, c), \ldots, \exs (n - 1, c)\}$$ since $i < j \leq 2^k < n$. Therefore, $m$ fails to discriminate this set. Since this applies for all $m < 2^{k + 1}$, we have $D_{\exs}(n) \geq 2^{k + 1}$.

Since we have $2^{k + 1} \leq D_{\exs} \leq 2^{k + 1}$, this means that $D_{\exs} (n) = 2^{k + 1}$ and thus $D_{\exs} (n) = 2^{\lceil \log_2 n \rceil}$, provided that $a = 1$.

Even for $a \neq 1$, we observe that the value of $2^{\lceil \log_2 n \rceil}$ is a power of 2 for all $n$, and so it is co-prime to all odd $a$. Therefore, we can apply Lemma~\ref{coprimescaledisc} to prove that the discriminator remains unchanged for odd values of $a$, thus proving that the discriminator for the sequence, $(\exs(n, c))_{n \geq 0} = \left(a \frac{(t^2)^{n + c} - 1}{2^b}\right)_{n \geq 0}$ is $D_{\exs} (n) = 2^{\lceil \log_2 n \rceil}$.
\end{proof}
\end{theorem} 

\section{Final remarks}

We have considered sequences of the form $(\ex(n))_{n \geq 0} = \left(a \frac{(t^2)^n - 1}{2^b}\right)_{n \geq 0}$ for odd integers $a$ and $t$, where $b$ is the smallest positive integer such that $t \not\equiv \modd{\pm 1} {2^b}$. We showed that the discriminator for this sequence is characterized by $D_{\ex} (n) = 2^{\lceil \log_2 n \rceil}$ and that the discriminator is shift-invariant, i.e., all sequences of the form $(\ex(n + c))_{n \geq 0}$ for $c \geq 0$ share the same discriminator.

This raises the question of what other sequences have shift-invariant
discriminators. 
It is easy to show that sequences defined by a linear equation, i.e. of the form $(an + b)_{n \geq 0}$, have shift-invariant discriminators. 
Furthermore, the first author has recently shown \cite{Haque17} that 
the sequence
$(2^k cn^2 +  bcn)_{n \geq 0}$, for a positive integer $k$ and odd integers
$b, c$, also has a shift-invariant discriminator $2^{\lceil \log_2 n \rceil}$. 

It is an open problem as to whether there are any sequences, other than those mentioned here, whose discriminators are shift-invariant. Futhermore,
all sequences whose discriminators are known to be shift-invariant have discriminators with linear growth, 
but we do not know if this is true of all shift-invariant discriminators.

\section{Acknowledgments}

We are grateful to Pieter Moree for introducing us to this interesting topic of discriminators. He also suggested the idea of generalizing $t$ to be any positive odd integer, thus broadening the class of exponential sequences presented in this paper.

\end{document}